\documentclass{amsart}
\usepackage{graphicx} 
\usepackage{amsfonts,amsmath,tikz,mathabx,caption}
\usepackage{amssymb}
\usepackage{pb-diagram}
\usepackage[all]{xy}
\usepackage{hyperref,xcolor}
\usetikzlibrary{decorations.markings}
\usetikzlibrary{shapes}
\usetikzlibrary{backgrounds}
\usepackage{subfig}

  \newtheorem{thm}{Theorem}

\numberwithin{thm}{section}

\title{Just infinite quotients of finitely generated subgroups of $\textup{PL}^+[0,1]$.}
\begin{document}
\author{Yash Lodha}

\date{\today}

\begin{abstract}
We show that just infinite quotients of finitely generated subgroups of Richard Thompson's group $\textup{F}$ are virtually abelian, answering a question of Grigorchuk. We show the same holds for the group of piecewise linear orientation preserving homeomorphisms of the interval, and the group of piecewise projective homeomorphisms of the real line. The latter provides a plethora of examples of nonamenable groups with this property. 
\end{abstract}
\maketitle

A group is said to be \emph{just infinite} if every nontrivial normal subgroup is of finite index.
A straightforward application of Zorn's lemma to the inclusion ordered poset of infinite index normal subgroups shows that every finitely generated infinite group admits a just infinite quotient.
Wilson showed that the study of just-infinite groups can be reduced to the study of simple groups, branch groups, and hereditarily just-infinite groups, i.e. groups all of whose finite index subgroups are just-infinite (see \cite{GrigorchukJI}, \cite{Wilson}).

The following definitions are due to Grigorchuk, who communicated these to the author in private correspondence.
A group is said to be \emph{elementarily just infinite} if every just infinite quotient is virtually abelian. 
A group is said to be \emph{locally elementarily just infinite}, or in the class (LEJI), if every finitely generated subgroup is elementarily just infinite.

Recall that $\textup{PL}^+[0,1]$ is the group of piecewise linear orientation preserving homeomorphisms of $[0,1]$. Thompson's group $\textup{F}\leq \textup{PL}^+[0,1]$ is the subgroup consisting of elements that satisfy that the slopes, whenever they exist, lie in $\{2^n\mid n\in \mathbf{Z}\}$ and whose breakpoints lie in $\mathbf{Z}[\frac{1}{2}]\cap [0,1]$. The group $\textup{H}(\mathbf{R})$ is the group of piecewise projective homeomorphisms of the real line defined in \cite{monod:pp1}, and the group $G_0$ is a finitely presented nonamenable group without free subgroups constructed by the author with Justin Moore in \cite{LodhaMoore} (shown to be of type $F_{\infty}$ in \cite{Lodha}), which emerges as a subgroup of $\textup{H}(\mathbf{R})$.
The goal of this note is to show the following:

\begin{thm}\label{main}
The groups $\textup{PL}^+[0,1]$ and $\textup{H}(\mathbf{R})$ lie in the class \emph{(LEJI)}.
\end{thm}
Since the class (LEJI) is closed under taking subgroups, Thompson's group $\textup{F}$ and the finitely presented nonamenable group $G_0$ also lie in this class. Whether this is the case for Thompson's group $F$ was a question recently posed to the author by Rostislav Grigorchuk in private correspondence.
Since $\textup{H}(\mathbf{R})$ and many of its subgroups are nonamenable \cite{monod:pp1}, we provide many examples of nonamenable groups in the class (LEJI). Note that groups which contain nonabelian free subgroups do not belong in this class (in particular, Thompson's groups $\textup{T},\textup{V}$), and neither do Tarski monsters or finitely generated infinite torsion groups such as Burnside groups or Golod-Shafareyvich groups.
The following question of Grigorchuk is also an interesting direction of research: Is there an (LEJI) group of intermediate growth? 

\section{The proof.}
All actions will be right actions.
Given an element $f\in \textup{PL}^+[0,1]$ and a subgroup $H\leq \textup{PL}^+[0,1]$ recall that the supports of $f$ and $H$ are defined as: $$Supp(f)=\{x\in [0,1]\mid x\cdot f\neq x\}\qquad Supp(H)=\bigcup\{Supp(h)\mid h\in H\}.$$ 
Given a group $G$, we denote by $G^{(n)}$ the $n$'th derived subgroup of $G$.

\begin{proof}[Proof of Theorem \ref{main}]
Given a finitely generated subgroup $H$ of $\textup{PL}^+[0,1]$, there exist finitely many points: $$0=x_0< x_1< ... <x_n=1$$ such that
for each $0\leq i\leq n-1$ either $H$ fixes each point in $[x_i,x_{i+1}]$ or each point in $(x_i,x_{i+1})$ is moved by some element in $H$. 
Note that every element $h$ of the derived subgroup satisfies that: $$\overline{Supp(h)}\subseteq \bigcup_{0\leq i\leq n-1}(x_i,x_{i+1})$$ 
Let $\phi:H\to J$ be a surjective homomorphism so that $J$ is just infinite. 
Since $J$ is finitely generated
so is $J^{(1)}$, since $J^{(1)}$ is a finite index subgroup in $J$ (unless $J^{(1)}$ is trivial, in which case $J$ is abelian and we are done).
Since each generator in a finite generating set for $J^{(1)}$ is a product of commutators of elements in $J$, there are elements $g_1,...,g_k\in H^{(1)}$ such that
$\phi(\langle g_1,...,g_k\rangle)=J^{(1)}$.

We shall construct a sequence of groups: $$H\geq G_0\geq G_1\geq G_2\geq \ldots \geq G_l$$ for some $0\leq l\leq n$ such that $\phi(G_i)=J^{(i+1)}$ for each $1\leq i\leq l$,
and so that $\phi(G_l)=J^{(l+1)}$ is non-trivial and abelian. Since each $J^{(i)}$, whenever nontrivial, is of finite index in $J$, we shall conclude that $J$ is virtually abelian.

We fix $G_0=\langle g_1,...,g_k\rangle$, and recall that $\phi(G_0)=J^{(1)}$.
Define $L_0=Supp(G_0)$ and note that: $$\overline{L_0}\subset \bigcup_{0\leq i\leq n-1}(x_i,x_{i+1})$$
Note that if $L_0$ is empty then $G_0$ is trivial and hence $J^{(1)}$ is trivial, implying that $J$ is abelian.
Assume that $L_0$ is nonempty.
Let $I_0=L_0\cap [x_j,x_{j+1}]$ for the smallest $j$ such that $L_0\cap [x_j,x_{j+1}]$ is nonempty. 
Since $H$ does not fix a point in $(x_j,x_{j+1})$, we can choose $k_0\in H$ such that $sup(I_0)\cdot k_0<inf(I_0)$ and hence $(I_0\cdot g_0)\cap I_0=\emptyset$.
Let $G_1=\langle [G_0^{k_0},G_0]\rangle$. Note that here $[G_0^{k_0},G_0]$ is the set $\{[\alpha,\beta]\mid\alpha\in G_0^{k_0},\beta\in G_0\}$ and $G_1$ is the group generated by it. It follows that $\phi(G_1)=\langle[\phi(G_0)^{\phi(k_0)},\phi(G_0)]\rangle $.
Note that since $\phi(G_0)=J^{(1)}$, it is characteristic in $J$, and hence $\phi(G_0)^{\phi(k_0)}=\phi(G_0)=J^{(1)}$.
It follows that $\phi(G_1)=[J^{(1)},J^{(1)}]=J^{(2)}$.
Moreover, note that $\overline{Supp(G_1)}\subset \bigcup_{j+1\leq i\leq n-1}(x_i,x_{i+1})$.

We continue inductively in this fashion, as follows. Having constructed $G_j$ such that $\phi(G_j)=J^{(j+1)}$, and if $G_j$ is nontrivial,
let $Supp(G_j)=L_j$ and $I_j=L_j\cap [x_p,x_{p+1}]$ for the smallest $p$ such that $L_j\cap [x_p,x_{p+1}]$ is nonempty.
We find $k_j\in H$ such that $(I_j\cdot k_j)\cap I_j=\emptyset$ and define $G_{j+1}=\langle [G_j^{k_j},G_j]\rangle$.
It follows that since $\phi(G_j)=J^{(j+1)}$ it is characteristic in $J$, we have that: $$\phi(G_{j})^{\phi(k_j)}=\phi(G_j)=J^{(j+1)}$$
It follows that: $$\phi(G_{j+1})=\langle [\phi(G_j)^{\phi(k_j)},\phi(G_j)]\rangle=\langle [\phi(G_j),\phi(G_j)]\rangle =J^{(j+2)}.$$
It is clear that this process ends with a group $G_l$ for some $l\leq n$ such that $[G_l^{k_l},G_l]$ is trivial.
And hence $\phi(G_l)=J^{(l+1)}$ is abelian.

The proof in the case of $\textup{H}(\mathbf{R})$ is similar to the proof above, except that the germs at fixed points of the subgroup $H$ are metabelian rather than abelian. So we start with the second derived subgroup $J^{(2)}$ of the just infinite quotient $J$, rather than the derived subgroup, and follow a similar argument as above.
\end{proof}

\bibliographystyle{amsalpha}
\bibliography{bib}

\end{document}